\documentclass{amsart}[11pt]

\usepackage{amsmath}
\usepackage{amssymb}
\usepackage{amsthm}
\usepackage[mathscr]{eucal}
\usepackage{mathrsfs}
\usepackage{psfrag}
\usepackage{fullpage}
\usepackage{color}
\usepackage{eucal}
\usepackage[dvips]{graphicx}
\usepackage{amsfonts}%
\usepackage{amsaddr}
\usepackage{graphicx}
\usepackage{graphics}

\newtheorem{theorem}{Theorem}
\newtheorem{example}{Example}
\newtheorem{definition}{Definition}

\newtheorem{proposition}{Proposition}

\newtheorem{lemma}{Lemma}

\newcommand{\R}{\mathbb{R}}
\newcommand{\C}{\mathbb{C}}

\newcommand{\1}{\mathbf{1}}
\newcommand{\mL}{\mathcal L}
\newcommand{\mA}{\mathcal A}
\newcommand{\mS}{\mathcal S}
\newcommand{\mB}{\mathcal B}

\begin{document}

\title{Improving the convergence of reversible samplers}

\author{Luc Rey-Bellet}
\address{Department of  Mathematics and Statistics\\
University of Massachusetts Amherst, Amherst, MA, 01003}
\email{luc@math.umass.edu}

\author{Konstantinos Spiliopoulos}
\address{Department of  Mathematics and Statistics\\
Boston University, Boston, MA, 02215}
\email{kspiliop@math.bu.edu}
\thanks{
K.S. was partially supported by the National Science Foundation (NSF)
DMS 1312124 and during revisions of this article by NSF CAREER award DMS 1550918. LRB was partially supported by the NSF DMS 1109316.
}
\date{\today}

\maketitle

\date{\today}

\begin{abstract} In Monte-Carlo methods the Markov processes used to sample a given target distribution
usually satisfy detailed balance, i.e.  they are  time-reversible. However, relatively recent results have
demonstrated that appropriate  reversible and irreversible perturbations can accelerate convergence to
equilibrium. In this paper we present some general design principles which apply to general Markov
processes.  Working with the generator of Markov processes, we prove that for some of the most commonly
used performance criteria,  i.e., spectral gap, asymptotic variance and large deviation functionals,
sampling is improved for  appropriate reversible and irreversible perturbations of some initially
given reversible sampler. Moreover we provide specific constructions for such reversible and irreversible
perturbations for various commonly used  Markov processes, such as Markov chains and diffusions.
In the case of diffusions, we make the discussion more specific using the large deviations rate function as a
measure of performance.
\end{abstract}

\textbf{Keywords:} Markov processes, Monte Carlo Sampling, Irreversibility, Detailed balance, Langevin Sampling, Large deviations, Asymptotic Variance

\section{Introduction}\label{S:Introduction}
In this paper we study the problem of sampling from a probability distribution $\pi(dx)$ which, typically,  is
known only up  to a normalizing constant. Sampling directly from $\pi(dx)$ is often infeasible and thus one
needs to rely on approximations. For example if $f:E\mapsto \R$ is a given observable on the state space $E$
and if one is interested in computing $\bar{f}=\int_{E}f(x)\pi(dx)$ one  constructs a positive recurrent  Markov process $X(t)$ which has $\pi$ as its invariant distribution.  Using the ergodic theorem
\[
\lim_{t\rightarrow\infty}\frac{1}{t} \int_{0}^{t}f(X(s))ds=\bar{f}, \text{a.s. for }f\in L^{1}(\pi).
\]
one can approximate $\bar{f}$ by  $f_{t}=\frac{1}{t} \int_{0}^{t}f(X(s))ds$ for sufficiently large $t$.
Clearly  the degree to which such an approximation is efficient depends on the ergodic properties
of the Markov process $X(t)$ and on the criterion used for comparison.

Many different reversible and irreversible algorithms have been proposed in the literature
dealing with both discrete and continuous (time or space) Markov chains as well
as diffusion processes. For Markov chains we refer the reader to \cite{AtheyaDossSethuraman1996, Bierkens2015,ChenHwang2013, DiaconisHolmesNeal2010, DiaconisMiclo2013, FrigessiHwangYounes1992,
FrigessiHwangSheuStefano1993,LeisenMira2008, Mira2001a, Mira2001b, MiraGeyer2000, Neal2004,
Peskun1973, SunGomezSchmidhuber2010,Tierney1998}
and for diffusion processes we refer the reader to
\cite{DuncanLelievrePavliotis2015,HwangMaSheu1993,HwangMaSheu2005,HwangNormandSheu2014,LelievreNierPavliotis2012,ReyBelletSpiliopoulos2014,ReyBelletSpiliopoulos2014b}. In most of these works, a
reversible Markov chain or diffusion,
$X_{0}(t)$, that has $\pi$ as its invariant distribution is taken as a reference process and then different
reversible or irreversible perturbations are explored which maintain the
same invariant measure and lead to improved sampling properties. The criteria that are mostly used
for comparison purposes are the spectral gap of the generator of the process and the asymptotic variance of
the estimator. Relatively recently, the large deviations rate function has been proposed  in  \cite{DupuisDoll1}
and used in \cite{DupuisDoll1,ReyBelletSpiliopoulos2014,ReyBelletSpiliopoulos2014b} as an alternative
criterium for  convergence and its connection to the asymptotic variance  have been explored.

The contribution of this paper is threefold. Firstly, we unify and extend existing results in the literature
demonstrating  that
there is a general underlying principle that applies to virtually all appropriate modifications of given reversible
Markov processes, without having to restrict attention to continuous or discrete Markov jump processes
or diffusion processes. Working directly  with the infinitesimal generator of the Markov process, we prove that, under suitable conditions,
reversible perturbations by negative definite generators as well as irreversible perturbations that maintain the
invariant measure result in faster convergence to equilibrium. We prove that this is true based on all commonly
used criteria of convergence; spectral gap, asymptotic variance and large deviations. We remark however that in this paper we restrict attention to additive perturbations of a generator by assymetric and anti-symmetric operators and we do not discuss techniques such as importance sampling, splitting, stratification and sequential sampling.

Secondly, we discuss specific constructions of such reversible and irreversible perturbations. We focus on
continuous time Markov chains, Markov jump processes and diffusion processes.
Some of these specific constructions are known in the literature,
such as the Peskun and Tierney constructions, \cite{Peskun1973, Tierney1998}, whereas others are novel,
such as the reversible perturbation of Markov jump processes, Example \ref{Ex:ReversibleJumpProcess}, the
reversible perturbation of diffusion processes, Example \ref{Ex:ReversibleDiffusion}
and the irreversible perturbations of generic Markov chains, Example \ref{IrreversibleMC}.

Thirdly, following \cite{DupuisDoll1,ReyBelletSpiliopoulos2014,ReyBelletSpiliopoulos2014b} we argue that large
deviations is a natural criterion for comparison for ergodic averages since it looks directly at the
actual numerical approximation, which is the ergodic average.
It has the advantage that in many cases it allows explicit computations which helps when comparing different
algorithms. Also, it is directly connected to the asymptotic variance, as the second order Taylor expansion
of the large deviation rate function around the
limit $\bar{f}$ is inversely proportional to the asymptotic variance. We focus mainly on diffusion processes
where the known form of
the rate function allows to do comparisons among specific algorithms.

The rest of the paper is organized as follows. In Section \ref{S:Examples}, we discuss the type of allowed perturbations and provide specific examples of such possible perturbations for
cases of interest, such as Markov chains, Markov jump processes and diffusion processes. In Section \ref{S:MainLemmas} we prove that the previously mentioned perturbations
 lead to improvement of sampling based on the behavior of the spectral gap, the asymptotic variance and of the large deviations rate function.
In Section \ref{S:LDPanalysisMarkovChain} we study some consequences of our theory for irreversible perturbations of Markov Chains.
Moreover, using large deviations, in Section \ref{S:LDPanalysisDiffusions}, we study the effect of appropriate negative reversible and irreversible perturbations of reference reversible diffusion processes on the rate of convergence to equilibrium.

\section{Perturbations of reversible Markov processes}\label{S:Examples}
Let us consider an ergodic time reversible continuous-time Markov process $X_0(t)$ on the state space $K$
with invariant measure $\pi$.  Let $L^2_{\R}(\pi)$ be the real Hilbert space with scalar product
$\langle f, g \rangle = \int f(x)  g(x) \pi(dx)$.  We denote by $T_{t}^{0}$ the corresponding strongly continuous
Markov semigroup as an operator
on $L^2_{\R}(\pi)$  with infinitesimal  generator $\mL_0$ with domain $D(\mL_0)$.   When discussing spectral properties we
will need also to consider $L^2_{\C}(\pi)$, the complex Hilbert space with scalar product
$\langle f, g \rangle = \int f(x) \bar{g}(x) \pi(dx)$.  All operators involved here are real operators (they map real functions into real
functions) and so they extend trivially to $L^2_{\C}(\pi)$. Abusing notation slightly, we will use the same notation for the operators
acting on the real or complex Hilbert spaces.

Since $X_0(t)$ is time-reversible, $T_{t}^{0}$ and its generator $\mL_0$ are self-adjoint: that is we have
\begin{equation}
\langle f, \mL_0 g \rangle  =  \langle \mL_0 f,  g \rangle
\end{equation}
for all $f,g  \in D(\mL_0)$.

We shall also assume the semigroup $T_{t}^{0}$ has a spectral gap in $L^2_{\R}(\pi)$, i.e., there exists
$\lambda_0 < 0$ such that
\begin{equation}
\sigma(\mL_0) \setminus \{ 0\} \,\subset\, (-\infty,  \lambda_0]
\end{equation}
where $\sigma(\mL_0)$ denotes the spectrum.  Note that $\mL_0$ is then  negative definite, i.e. we have
\begin{equation}
\langle f,  \mL_0 f \rangle  \le 0
\end{equation}
for all $f$ in $L_{\C}^2(\pi)$.

We may think of $X_0(t)$ as a reference process and we now introduce two types of ``perturbations"  $X(t)$ of
the processes $X_0(t)$ where we require that $X(t)$ has the same invariant measure $\pi$. In the first type of perturbation $X(t)$  maintains the reversibility property, even though the dynamics have changed,   whereas in the second type of perturbation $X(t)$ is no longer reversible.

We describe then simple criteria which ensure that the process $X(t)$ converges faster to
equilibrium than $X_0(t)$ in various senses. In Section \ref{S:MainLemmas} we prove that these perturbations lead to faster convergence to equilibrium. Essentially, reversible perturbations by negative definite operators and irreversible perturbations, that is perturbations by adding an antisymmetric part to the generator,
lead to improvement in sampling.

In Subsection \ref{SS:Reversible} we look at reversible perturbations and Examples \ref{Ex:ReversibleMC}-\ref{Ex:ReversibleDiffusion} present some concrete constructions. Then,
in Subsection \ref{SS:Irreversible} we look at irreversible perturbations and
 Examples
\ref{IrreversibleMC}-\ref{Ex:IrreversibleDiffusion} present some related exact constructions.

\bigskip
\subsection{Reversible perturbations}\label{SS:Reversible}  We consider a Markov process with generator $\mL=\mL_0  + \mS$ and we assume that
\begin{enumerate}
\item We have $D(\mL_0) \subset D(\mS)$
and $\mL_0 + \mathcal S$ is the generator of a Markov process with invariant measure $\pi$
\item
For all $f, g \in D(\mS)$ we have
\begin{equation*}
\langle f ,\mS g \rangle \,=\,  \langle \mS f, g \rangle
\end{equation*}
i.e. $\mS$ is self-adjoint.  This implies that both $\mL$ and $\mL_0$ are self-adjoint.
\item  $\mS$ is {\it negative definite} i.e.,
\begin{equation*}
\langle f , \mS f \rangle \le 0
\end{equation*}
for all $f \in D(\mS)$.
\end{enumerate}

Let us see now some specific examples of Markov processes where the perturbation $\mS$ can be constructed.
\begin{example}{\rm
{\bf (Markov chains on finite discrete spate space, Peskun condition).} \label{Ex:ReversibleMC} Consider a continuous-time
Markov chain on a discrete finite state space $K=\{1, \cdots, N\}$  with transition probability kernel $k_0(i,j)$.  The perturbation
$\mS$ is such that  for all pairs $i,j$ in $\mS$ with $i\not =j$ we have for the new transition probability kernel
\begin{equation}
k(i,j)\ge k_0(i,j)\,.
\end{equation}

This means that the jump rate for $X(t)$ is bigger than the jump rate for $X_0(t)$ for any part of the state.  Intuitively it means that
the Markov chains spends less time in its current state and this should speed up the convergence.  This condition was introduced
in \cite{Peskun1973} for discrete-time Markov chain and in \cite{LeisenMira2008} for  continuous time and shown to lead to decreased variance.

Let us discuss now how one can construct $\mS$ concretely. Since we require both $X$ and $X_0$ to have the same invariant measure clearly $\mS(i,j)$ and $\mS(j,i)$ are not independent.
But, for two different pairs of states $(i,j)$ and $(i',j')$ we can choose $\mS(i,j)$ and $\mS(i', j')$ completely independently. Thus, we can
write
\begin{equation*}
\mS = \sum_{1 \le i < j \le N}  \mS^{(i,j)}
\end{equation*}
where $\mS^{(i,j)}$ has the form
\begin{equation*}
 \left( \begin{array}{ccccc}
               &                  &               &               &     \\
               &-\epsilon&  \cdots   & \epsilon &   \\
               &   \vdots  &      \vdots & \vdots   &    \\
               &   \delta  &      \cdots & -\delta  &    \\
              &                  &                &               &     \\
\end{array}\right)
\end{equation*}
and $\delta$ and $\epsilon$ are non-negative and satisfy
\begin{equation}\label{Scond}
\pi(i) \epsilon = \pi(j) \delta.
\end{equation}

The entries in $\mS^{(i,j)}$ are all zeros apart from the $i,j$ rows and columns  where the indicated values are taken. Condition \eqref{Scond} ensures that $\mS$ is self-adjoint since
for any $f=(f(1), \cdots f(N))^T$ and  $g=(g(1), \cdots g(N))^T$ in $L^2_\R(\pi)$ we have
\begin{eqnarray*}
<f, \mS^{(i,j)} g> \,&=&\,  \pi(i) f(i) ( - \epsilon g(i) + \epsilon g(j))  -  \pi(j) f(j) ( - \delta g(i) + \delta g(j))) \nonumber \\
\,&=&\,  - \epsilon \pi(i) (f(i) - f(j))(g((i)-g(j))
\end{eqnarray*}
which is obviously symmetric in $f$ and $g$.  So, $\mS$ is self-adjoint on $L^2_\R(\pi)$ and then also on $L^2_\C(\pi)$.
This  ensures that both processes $X$and $X_0$ satisfies detailed balance with
respect to $\pi$.  In addition we have for any $f=(f(1), \cdots f(N))^T \in L^2_\R(\pi)$
\begin{eqnarray*}
<f, \mS^{(i,j)} f> \,&=&\,  - \epsilon \pi(i) (f(i) - f(j))^2 \le 0
\end{eqnarray*}
hence $\mS$ is negative definite since $\epsilon$ and $\delta$ are non-negative.

}
\end{example}

\begin{example}{\rm
{\bf (General jump process).}\label{Ex:ReversibleJumpProcess} Let us consider a continuous time Markov jump process that has bounded infinitesimal generator taking values on a state space $K$. The general form of its generator takes the form
\begin{equation*}
\mathcal{L}_{0}f(x)=\lambda(x)\int_{K}\left(f(y)-f(x)\right)\alpha(x,dy)
\end{equation*}
where $\lambda$ is a nonnegative bounded intensity function on $K$ and $\alpha(x,\Gamma)$ is a transition kernel on $K\times\mathcal{B}(K)$.

The construction of such a jump process can be done as follows. Consider a Markov chain $X_{n}$ on $K$ with transition probability $\alpha(x,\Gamma)$ and letting
$\tau_{1}, \tau_{2},\cdots$ be independent (between them and from $X_{n}$ for every $n\in\mathbb{N}$ as well) and exponentially distributed random variables with mean $1$,
define $s_{\kappa}$ via the relation $\lambda(X_{\kappa-1})s_{\kappa}=\tau_{\kappa}$. Then, the Markov jump process with generator $\mathcal{L}_{0}$ is given by
\begin{equation}
X_{0}(t)=X_{n},\quad \text{for}\quad \sum_{\kappa=1}^{n}s_{\kappa}\leq t<\sum_{\kappa=1}^{n+1}s_{\kappa}.\label{Eq:JumpMarkovProcess}
\end{equation}

Let us assume that there exist $0<\lambda_{1}\leq \lambda_{2}<\infty$ such that for all $x$, $\lambda_{1}\leq \lambda(x)\leq \lambda_{2}$.
Then, under  appropriate conditions on the transition kernel $\alpha$, see for example Section 2 of \cite{DupuisYufei2013}, we have that $X_{0}(t)$ is an ergodic process. In particular, $\alpha$ has then an invariant distribution denoted by $\tilde{\pi}$ and the boundedness of $\lambda(\cdot)$ allows us to define
\[
\pi(E)=\frac{\int_{E}\frac{1}{\lambda(x)}\tilde{\pi}(dx)}{\int_{K}\frac{1}{\lambda(x)}\tilde{\pi}(dx)}
\]
which can be shown to be the unique invariant distribution of $X_{0}(t)$. Now, we also make the assumption that the process $X_{0}$ is reversible, which means that for all $x,y\in K$
\begin{equation}
\lambda(x)\alpha(x,dy)\pi(dx)=\lambda(y)\alpha(y,dx)\pi(dy).\label{Eq:DetailedBalanceJumpProcess}
\end{equation}

There are many different reversible perturbations that one can imagine. Perhaps the simplest one is to use the Peskun-Tierney \cite{Peskun1973, Tierney1998} construction on the Markov chain $X_{n}$ that
is used to define
the jump Markov process $X_{0}(t)$ via (\ref{Eq:JumpMarkovProcess}), as follows. Notice that we can write
\begin{equation*}
\mathcal{L}_{0}f(x)=\int_{K}f(y)A(x,dy)
\end{equation*}
where setting $\|\lambda\|=\sup_{x\in K}\lambda(x)>0$, we have defined
\[
A(x,dy)=\nu\|\lambda\|\left(\hat{\alpha}(x,dy)-\delta_{x}(dy)\right),\quad\text{ and }\quad \hat{\alpha}(x,dy)= \frac{\lambda(x)}{\|\lambda\|}\alpha(x,dy)+\left(1-\frac{\lambda(x)}{\|\lambda\|}\right)\delta_{x}(dy).
\]

Let us now consider a transition probability operator $\beta(x,dy)$  such that
for almost every $x\in K$, $\beta(x,\Gamma\setminus\{x\})\geq \alpha(x,\Gamma\setminus\{x\}) $ for every $\Gamma\in\mathcal{B}(K)$. Assume that $\beta(x,dy)$ is such that
(\ref{Eq:DetailedBalanceJumpProcess}) holds and consider the jump Markov process with generator
\begin{equation*}
\mathcal{L}f(x)=\int_{K}f(y)B(x,dy)
\end{equation*}
where $B(x,dy)$ is as $A(x,dy)$ with $\beta(x,dy)$ in place of $\alpha(x,dy)$.

Now, we are in the set-up of \cite{Tierney1998}. It is easy to see that for almost every $x\in K$ we have that
$B(x,\Gamma\setminus\{x\})\geq A(x,\Gamma\setminus\{x\}) $ for every $\Gamma\in\mathcal{B}(K)$.  Lemma 3 in \cite{Tierney1998} guarantees that the operator $\mL-\mL_{0}=\mS$ is negative operator in $\mathcal{L}^{2}(\pi)$. In Section \ref{S:MainLemmas} we prove that if one uses the jump Markov process with generator $\mathcal{L}f(x)$ instead of $\mathcal{L}_{0}f(x)$, then the sampling properties of the algorithm are better.
}
\end{example}

\begin{example}{\rm
{\bf (Diffusions with multiplicative noise).}\label{Ex:ReversibleDiffusion} Let $T>0$ and consider the diffusion on  $\R^d$
\begin{equation}\label{GeometricLangevin}
dX(t) \,=\,  \left[ - \Sigma( X(t)) \nabla U( X(t))  + T\nabla \cdot \Sigma (X(t))\right] dt  + \sqrt{2 T} \sigma( X(t)) dB(t)
\end{equation}
where $B$ is a $d$-dimensional Brownian motion,  $\sigma : \R^d \to  \R^{d\times d}$,
$\Sigma(x) = \sigma(x) \sigma(x)^T$ and $\nabla \cdot \Sigma$ denotes the vector field with components
$\sum_{j} \partial_{x_j} \Sigma_{i,j}(x)$.

If $\sigma$ is the identity matrix the equation reduces to the standard overdamped Langevin equation
\begin{equation}\label{Langevin}
dX_0(t) \,=\,  -   \nabla U(X_0(t)) dt  + \sqrt{2 T} dB(t)
\end{equation}
which we take as our reference process.  The generator $\mL$ is given by
\begin{equation*}
\mL \,=\,  T \nabla \cdot \Sigma \nabla - \Sigma \nabla U \cdot \nabla
\end{equation*}
In any case under suitable regularity and growth conditions  on $U$ and $\sigma$ the process $X(t)$ is ergodic
and the measure
\begin{equation*}
\pi(dx) = Z^{-1} e^{ - U(x)/T} dx, \qquad \text{with } Z=\int e^{ - U(x)/T} dx
\end{equation*}
is invariant for \eqref{GeometricLangevin} and $X(t)$ is reversible.  We have for $f,g \in D(\mL)$
\begin{equation*}
\langle f, \mL g \rangle \,=\,  - T\int  \nabla f(x) \cdot \Sigma(x) \nabla g(x) \pi(dx)\,.
\end{equation*}

Given that for $f,g \in D(\mL_{0})$ the reference generator $\mL_{0}$ satisfies
\begin{equation*}
\langle f, \mL_{0} g \rangle \,=\,  - T\int  \nabla f(x) \cdot \nabla g(x) \pi(dx)\,.
\end{equation*}
we get that  the perturbation $\mS$ has the form
\begin{equation*}
\langle f, \mS g \rangle \,=\,  - \int  \nabla f(x) \cdot (\Sigma(x)- \1) \nabla g(x) \pi(dx) \,.
\end{equation*}
A convenient choice is to take
\begin{equation*}
\sigma(x) \,=\,  \1 + A(x)
\end{equation*}
where we choose $A$ such that $A+A^T$ is nonnegative definite. Then we have
\begin{equation*}
\Sigma= \1 + A + A^T + AA^T
\end{equation*}
and $\mS$ is negative definite. In the context of Hamiltonian Monte Carlo, the authors in \cite{GirolamiCarderhead2011} suggest using (\ref{GeometricLangevin})
with a special choice for the  matrix $\Sigma(x)$. 
In Section \ref{S:MainLemmas} we prove that regular enough choices of $\Sigma(x)$ such that $\Sigma(x)-I$ is
positive definite, lead to improved sampling.
The degree of improvement depends of course on the choice of $\Sigma(x)$.
}
\end{example}

\bigskip
\subsection{ Irreversible perturbations}\label{SS:Irreversible}  We consider a Markov process with generator $\mL=\mL_0  + \mA$ and we assume that
\begin{enumerate}
\item We have
$D(\mathcal L) \subset D(\mA)$
and $\mathcal L + \mA$ is the generator of a Markov process with invariant measure $\pi$.
\item
For all $f, g \in D(\mA)$ we have
\begin{equation*}
\langle f , \mA g \rangle \,=\, -  \langle \mA f, g \rangle
\end{equation*}
i.e. $\mA$ is antiself-adjoint.   Clearly this implies that
\begin{equation*}
\langle f , \mA f \rangle = 0
\end{equation*}
for all (real-valued) $f \in D(\mA)$.
\end{enumerate}

\begin{example}
{\rm
{\bf (Markov chains on discrete state space).}\label{IrreversibleMC}  Consider a continuous-time Markov chain on a discrete finite state space $K=\{1,
\cdots, N\}$  with generator $\mL_0(i,j)$.  Comparisons of reversible and non-reversible Markov chains can be found in \cite{Bierkens2015,ChenHwang2013,DiaconisHolmesNeal2010,MiraGeyer2000,Neal2004}. Here we present a simple irreversible perturbation of a reversible Markov chain that leads to acceleration of convergence.

To construct a non-reversible perturbation consider a matrix $\Gamma(i,j)$ with
\begin{equation*}
\Gamma(i,j)=- \Gamma(j,i)\,, \quad  \sum_{j} \Gamma(i,j) =0
\end{equation*}
that is $\Gamma$ is antisymmetric and the sum of its rows (and columns) is $0$.  Then set
\begin{equation*}
\mA(i,j) = \frac{1}{\pi(i)} \Gamma(i,j)  \,.
\end{equation*}

We have then
\begin{equation*}
\sum_{i} \pi(i) \mA(i,j) \,=\, \sum_{i} \Gamma(i,j) =0
\end{equation*}
and this ensures that $\pi$ is the invariant measure for the generator $\mL=\mL_0 + \mA$.  Of course one needs to choose the entries in
$\Gamma$  sufficiently small such that the entries in $\mL$ are nonnegative. Moreover, the adjoint of $\mA$ on
$L^2_\R(\pi)$ is the matrix with entries $\mA^*(j,i)= \pi(i) \mA(i,j) \pi(j)^{-1}$ so that
\begin{equation*}
\mA^*(j,i)= \pi(i) \mA(i,j) \pi(j)^{-1} \,=\, \Gamma(i,j) \pi(j)^{-1} \,=\, - \pi(j)^{-1} \Gamma(j,i) \,=\, - \mA(j,i)
\end{equation*}
and thus $\mA$ is anti-selfadjoint.

To build concrete examples of such Markov chains we will express the perturbations in terms of
{\it cycles}.  To the reversible Markov chain  with generator $\mathcal{L}_0$ we associate,
in the usual manner, the undirected graph $G = (V,E)$ where the set of vertices $V=K$
and where the edge $(i,j)$ is in $E$ if $\mathcal{L}_0(i,j)>0$.  Now we can construct
irreversible perturbations in terms of {\it cycles} in the graph $G$.  If we assume for example
that the graph contains a cycle of length $3$, say,  through the states $i, j, k$ in $S$, then
we can take a perturbation $\Gamma^{(i,j,k)}$ to be of the form
\begin{equation*}
\Gamma^{(i,j,k)} \propto \left( \begin{array}{ccccccc}
               &                  &               &               &              &       &\\
               & 0              &  \cdots    & 1           & \cdots  & -1   &   \\
               &   \vdots    &      \vdots & \vdots  &  \vdots  &  \vdots    &\\
               &   -1           &      \cdots & 0        &  \cdots    & 1    &    \\
              &    \vdots    &    \vdots   &\vdots  & \vdots    & \vdots  & \\
               &   1           &      \cdots & -1        &  \cdots    & 0    &    \\
               &                  &               &               &              &       &
\end{array}\right)
\end{equation*}
where the $\cdots$ and $\vdots$ represent zero's and the elements shown are the $i,j,k$ rows and
columns.  The proportionality constant must be chosen small enough so that  the transition rates are
non-negative.   Then $\mA=\mA^{(i,j,k)}$ has the form
\begin{equation*}
\mA^{(i,j,k)}  \propto \left( \begin{array}{ccccccc}
               &                  &               &               &              &       &\\
               & 0              &  \cdots    & \frac{1}{\pi(i)}           & \cdots  & -\frac{1}{\pi(i)}   &   \\
               &   \vdots    &      \vdots & \vdots  &  \vdots  &  \vdots    &\\
               &   -\frac{1}{\pi(j)}           &      \cdots & 0        &  \cdots    & \frac{1}{\pi(j)}     &    \\
              &    \vdots    &    \vdots   &\vdots  & \vdots    & \vdots  & \\
               &   \frac{1}{\pi(k)}          &      \cdots & -\frac{1}{\pi(k)}         &  \cdots    & 0    &    \\
               &                  &               &               &              &       &
\end{array}\right)
\end{equation*}
To show how it can be achieved in a concrete Monte-Carlo situation consider the invariant measure
$\pi(i)=Z^{-1} e^{-H(i)}$:  any generator of the form
\begin{equation*}
\mathcal{L}(i,j) =  \frac{1}{\pi(i)} c(i,j)    \textrm{ with } c(i,j)=c(j,i)
\end{equation*}
is reversible with invariant measure $\pi(i)$.  Standard choices are the Glauber dynamics $\mathcal{L}_G$ and
the Metropolis dynamics $\mathcal{L}_M$ with
\begin{equation*}
\mathcal{L}_G(i,j) =  \frac{e^{H(i)}}{e^{H(i)} + e^{H(j)}} \,,  \quad \quad  \mathcal{L}_M(i,j) =  {e^{H(i)}} \min\left\{ e^{-H(i)}, e^{-H(j)} \right\}
\end{equation*}
which both do not depend on the, possibly hard to compute, normalization constant $Z$.

One easily constructs adapted irreversible perturbations which in turn do not depend on the normalization constant by choosing for example  for the Glauber dynamics
\begin{equation*}
\mA_G^{(i,j,k)} = \epsilon \left( \begin{array}{ccccccc}
               &                  &               &               &              &       &\\
               & 0              &  \cdots    & \frac{e^{H(i)}}{e^{H(i)} + e^{H(j)} + e^{H(k)}} & \cdots  & -\frac{e^{H(i)}}{e^{H(i)} + e^{H(j)} + e^{H(k)} }   &   \\
               &   \vdots    &      \vdots & \vdots  &  \vdots  &  \vdots    &\\
               &   - \frac{e^{H(j)}}{e^{H(i)} + e^{H(j)} + e^{H(k)}}           &      \cdots & 0        &  \cdots    & \frac{e^{-H(j)}}{e^{H(i)} + e^{H(j)} + e^{H(k)}}      &    \\
              &    \vdots    &    \vdots   &\vdots  & \vdots    & \vdots  & \\
               &   \frac{e^{H(k)}}{e^{H(i)} + e^{H(j)} + e^{H(k)}}          &      \cdots & - \frac{e^{H(k)}}{e^{H(i)} + e^{H(j)} + e^{H(k)}}        &  \cdots    & 0    &    \\
               &                  &               &               &              &       &
\end{array}\right)
\end{equation*}
and with the abbreviation $m(i,j,k) = \min \left\{ e^{-H(i)}, e^{-H(j)} , e^{-H(k)} \right\}$ for the Metropolis dynamics
\begin{equation*}
\mA_G^{(i,j,k)} = \epsilon \left( \begin{array}{ccccccc}
               &                  &               &               &              &       &\\
               & 0              &  \cdots    & {e^{H(i)}} m(i,j,k) & \cdots  & -{e^{H(i)}}  m(i,j,k)  &   \\
               &   \vdots    &      \vdots & \vdots  &  \vdots  &  \vdots    &\\
               &   - {e^{H(j)}} m(i,j,k)         &      \cdots & 0        &  \cdots    & {e^{H(j)}} m(i,j,k)   &    \\
              &    \vdots    &    \vdots   &\vdots  & \vdots    & \vdots  & \\
               &   {e^{H(k)}} m(i,j,k)       &      \cdots & - {e^{H(k)}}  m(i,j,k)       &  \cdots    & 0    &    \\
               &                  &               &               &              &       &
\end{array}\right)
\end{equation*}

Note that if we add perturbation for exactly one cycle we should take the coefficient $\epsilon$ sufficiently small so that the rates are non-negative. In general if we add perturbations for many,  or all, cycles the sum of the coefficients
of all cycles containing any given state should not add up to more than $1$ to ensure that the rates are non-
negative.  Also this can be generalized easily to cycles of arbitrary length and the details are left to
the reader.
}
\end{example}

\begin{example}{\rm
{\bf (Diffusions in $\R^{d}$ or on compact manifolds)}\label{Ex:IrreversibleDiffusion}
Consider the SDE
\begin{equation*}
dX(t) = - \nabla U(X(t)) + C(X(t))  + \sqrt{2T} dB(t)
\end{equation*}
in  $\mathbb{R}^d$ or in a compact manifold $E$. Assume that the growth properties of $U(x)$ and $C(x)$ are such that the SDE has a unique, non-explosive strong solution with a unique invariant measure. For $C(x)=0$ the process is reversible with invariant measure
\begin{equation*}
\pi(dx) \,=\, Z^{-1} e^{ -U(x)/T } dx,\qquad \text{where }Z=\int e^{ -U(x)/T } dx
\end{equation*}
and generator
\begin{equation*}
\mL_0 \,=\, T \Delta - \nabla U \cdot \nabla
\end{equation*}
and if we pick $C$ such that ${\rm div}(C(x)e^{- U(x)/T})=0$, then the invariant measure $\pi$ is maintained. Notice that
\begin{equation*}
\mL=\mL_{0}+\mA, \qquad\text{where } \mA \,=\, C \cdot \nabla
\end{equation*}
and then $A$ is antisymmetric in $L^2(\pi)$. The relation ${\rm div}(C(x)e^{- U(x)/T})=0$ is equivalent to ${\rm div}(C(x))=T^{-1} C(x)\cdot \nabla U(x)$, which is
implied if we assume that $C$ is divergence free and orthogonal to $\nabla U$, i.e., ${\rm div}(C(x))=0$ and $C(x)\cdot\nabla U(x)=0$. The results of
\cite{HwangMaSheu2005}, for spectral gap, and of \cite{ReyBelletSpiliopoulos2014,ReyBelletSpiliopoulos2014b} for
 the asymptotic variance and large deviations rate function, as well as  \cite{HwangNormandSheu2014,DuncanLelievrePavliotis2015}  for the asymptotic variance,
 show that the convergence improves when the irreversible perturbation $A$ is introduced.
}
\end{example}

\section{General theory on improvement of convergence properties}\label{S:MainLemmas}

We prove simple lemmas showing that perturbations of reversible and irreversible types ameliorate the
convergence properties of the algorithms for all commonly used criteria of convergence: spectral gap, asymptotic variance and large deviations rate function.

In specialized settings, versions of Lemmas \ref{L:SpectralGap} and \ref{L:AsymptoticVariance} below have appeared in the literature before, see
\cite{ChenHwang2013, DiaconisHolmesNeal2010, DiaconisMiclo2013, DuncanLelievrePavliotis2015, FrigessiHwangYounes1992, FrigessiHwangSheuStefano1993,HwangMaSheu1993,HwangMaSheu2005,
LeisenMira2008, Mira2001a, Mira2001b, MiraGeyer2000, Neal2004, Peskun1973, Tierney1998}. The novelty of  Lemmas \ref{L:SpectralGap} and \ref{L:AsymptoticVariance} is that working solely with the generator, we can prove in great generality (i.e., without restricting to specialized settings)  that perturbations of general reversible Markov processes
by negative reversible or irreversible generators decrease both spectral gap and asymptotic variance of the estimator.

Lemmas \ref{L:LargeDeviationsMeasures} and \ref{L:LargeDeviationsObservable}
state that the large deviations behavior is also improved. This is because  the tail probability of the estimator being away from the true value decreases faster, yielding faster convergence to equilibrium. This was studied in detail in \cite{ReyBelletSpiliopoulos2014} for the specific case of irreversible perturbations of reversible diffusion processes, i.e., in the setup of Example \ref{Ex:IrreversibleDiffusion}. Here we prove that this is true for Markov processes in general, without having to restrict attention to diffusion processes. We work directly with the generator of a given Markov process.

In the sequel we consider a generator of the type
\begin{equation*}
\mL= \mL_0 +  \mS + \mA
\end{equation*}
where $\mS$ is a reversible perturbation and $\mA$ is an irreversible one.

It will be useful to introduce the space
\begin{equation*}
H_\R^0 \equiv \{ f \in L^2_\R(\pi) \,;\,  \int f d \pi  = 0 \}
\end{equation*}
which is the subspace of $L^2_\R(\pi)$ which is orthogonal to the eigenspace corresponding to the eigenvalue $0$
of $\mathcal L$.   In particular $H_\R^0$ is invariant under the semigroup $T^t$.

In Lemma \ref{L:SpectralGap} we prove that the spectral gap associated to a Markov process with generator
$\mL$ is smaller than the spectral gap associated to a Markov process with generator $\mL_0$. In Lemma \ref{L:AsymptoticVariance} we prove that the asymptotic variance of the empirical average of a Markov process improves (i.e., decreases) under the reversible and irreversible perturbation. Lastly, in Lemmas \ref{L:LargeDeviationsMeasures} and \ref{L:LargeDeviationsObservable} we prove that a similar behavior is true from the eyes of the large deviations rate function for the empirical average.

\subsection{Spectral gap.}
Our first result is about the spectral gap which is defined in the general (non-reversible) case as
\begin{equation*}
\lambda \,=\,  \sup \{ {\rm Re}(z)\,;  z \in \sigma(\mL), z\not =0\} \,.
\end{equation*}

By the Hille-Philips theorem, see Section 12.3 \cite{HillePhillips}, the existence of a spectral gap  (i.e.  $\lambda <0$) implies a bound
\begin{equation*}
\| T_t f - \int f d\pi  \| \,\le \,  C e^{\lambda t} \| f - \int f d\pi \|
\end{equation*}
for all $f \in L^2_\R(\pi)$. Here $\|\cdot\|$ is the $L^2_\R(\pi)$ norm. Note that in the reversible case, i.e. when $T_t=T_t^0$ is associated with $\mL_{0}$, the spectral theorem implies that the constant $C$ is equal to $1$.

\begin{lemma}{\bf [Spectral gap].}\label{L:SpectralGap} The spectral gap $\lambda$ of the generator of semigroup with generator $\mL= \mL_0 + \mS + \mA
$  is smaller than the spectral gap $\lambda_0$ of $\mL_0$.
\end{lemma}

\proof We will use the fact that the reference operator $\mL_0$ is self-adjoint. Let $f \in D(\mL) \subset H^0_\R(\pi)$. Using that for real-valued $f$, $(f, \mA f) =0$ we then obtain
\begin{equation*}
 \frac{d}{dt} \| T_t f \|^2 \,=\, 2 \langle T_t f, (\mL_0 + \mS + \mA ) T_tf \rangle  \,=\,  2 \langle T_t f, (\mL_0 + \mS) T_tf \rangle   \,\le \,
2 \langle T_t f,  \mL_0 T_tf \rangle \,\le\, - 2 \lambda_0 \|T_t f\|^2 \,.
\end{equation*}
The latter implies that $\|T_t f\| \le e^{- \lambda_{0} t} \|f\|$ for any $f \in H_\R^0$. Notice that the pre-factor turns out to be $C=1$ here as well. But the norm of the real operator $T_t f$ acting $H_\R^0$
is the same as the norm on  $H_\C^0$. Since $\mL+\lambda_{0}$ generates a contraction semigroup, we have that $Re\left(\left<\left(\mL+\lambda_{0}\right)f,f\right>\right)\leq 0$, and so by Hille-Philips theorem we conclude that the spectrum of $T_t$
lies in the half-plane $\{ Re(z) \le - \lambda_0\}$.  \qed

\subsection{Asymptotic variance.}
We next turn to the asymptotic variance.  Let $f \in L^2_\R(\pi)$  be an observable and let $\bar{f} = \int f d\pi$. Note that
$f - \bar{f}  \in H_\R^0$.  We assume that the operators $\mL_0$ and $\mL$  are invertible when restricted to $H_\R^0$. We denote
by $\mL_0^{-1}$ and $\mL^{-1}$ their inverse which are bounded operators acting on $H_\R^0$.

For $f \in L^2$ let $S_t(f) = \int_0^t f(X(t)) \,dt$, then ${\bf E}_{\pi} (S_t(f))=t \bar{f}$ and the asymptotic variance of $S_t(f)/t$ satisfies
\begin{equation*}
\sigma^2(f) \equiv \lim_{t \to \infty} \frac{1}{t} {\rm Var}_{\pi}( S_t(f)/t) \,=\, 2 \int_0^\infty \langle T^t (f- \bar{f}) , (f - \bar{f}) \rangle \, dt \,=\,
\langle (f - \bar{f}) , (- \mL)^{-1} (f - \bar{f}) \rangle
\end{equation*}

Using this we prove in Lemma \ref{L:AsymptoticVariance} that the asymptotic variance never decreases by perturbations of the type $\mS+\mA$. Notice that in the case $\mS=0$ a similar result has been recently obtained in \cite{DuncanLelievrePavliotis2015} using different methods.
\begin{lemma}{\bf [Asymptotic variance].} \label{L:AsymptoticVariance}Let us assume that the operator $(-\mL- \mS)^{-1/2} \mA   (-\mL-\mS)^{-1/2}$ is bounded.
Then for any $f \in L^2_\R(\pi)$ we have
\begin{equation*}
\sigma^2(f) \le \sigma_0^2(f)
\end{equation*}
\end{lemma}

\proof  The reversible and irreversible perturbations use different arguments so we prove this in two steps.  We first compare
the variance for $\mL_0$ and $\mL_0 + \mS$.   We can restrict ourselves on the subspaces $H_\R^0$ where both operator are
invertible. Since $-\mL_0$  is positive definite it possess a square root and we write
\begin{equation*}
-\mL_0 - \mS \,=\, (-\mL_0)^{1/2} \left( \1 + (-\mL_0)^{-1/2} (-\mS) (-\mL_0)^{-1/2} \right)  (-\mL_0)^{1/2}
\end{equation*}
and thus
\begin{equation*}
(-\mL_0 - \mS)^{-1} \,=\, (-\mL_0)^{-1/2} \left( \1 + (-\mL_0)^{-1/2} (-\mS) (-\mL_0)^{-1/2} \right)^{-1}  (-\mL_0)^{-1/2}
\end{equation*}

By assumption $-S$ is non-negative definite therefore so is $T= (-\mL_0)^{-1/2} (-\mS) (-\mL_0)^{-1/2}$.
If we set   \begin{equation*}
g = (-\mL_0)^{1/2} (f - \bar{f})
\end{equation*}
the statement reduces to proving that for any $g$ we have
\begin{equation*}
\langle g ,  ( \1 + T)^{-1} g \rangle \le \langle g, g \rangle
\end{equation*}
But this follows immediately from the spectral theorem for self-adjoint operator.

To handle the irreversible perturbation let us consider a generator of the form $\mL_0 + \mA$ (if we have a symmetric perturbation
$\mS$ replace $\mL_0$ by $\mL_0+\mS$).  We notice first that any (bounded) operator $B$ can be written as a sum of a self-adjoint
part $(B+B^*)/2$ and an anti self-adjoint  part $(B- B^*)/2$.  Since $f$ is real-valued, only the self-adjoint part of the inverse of
$-\mL_0 - \mA$  matters in the asymptotic variance.  To compute we first write
\begin{equation*}
(-\mL_0 - \mA)^{-1} \,=\, (-\mL_0)^{-1/2} \left( \1 + (-\mL_0)^{-1/2} (-A) (-\mL_0)^{-1/2} \right)^{-1}  (-\mL_0)^{-1/2}
\end{equation*}

Set ${\mathcal B} \equiv(-\mL_0)^{-1/2} (-A) (-\mL_0)^{-1/2}$ which is anti-selfadjoint and thus
$(\1 + \mB)(\1 - \mB) = \1 - \mB^2 = \1 + \mB^*\mB$  we obtain
\begin{equation*}
(\1 + \mB)^{-1} \,=\,  (\1 +\mB^*\mB)^{-1}  - \mB (\1 +\mB^*\mB)^{-1}  \,.
\end{equation*}

Since $\mB^{*}=-\mB$ and $\mB$ commutes with $(\1 +\mB^*\mB)^{-1}$, we then have that $\langle g ,  \mB( \1 + \mB^*\mB)^{-1} g \rangle=0$. The latter implies that
\begin{equation*}
\sigma^{2}(f) \,=\, \langle (-\mL_0)^{1/2} (f - \bar{f}), (\1 +\mB^*\mB)^{-1} (-\mL_0)^{1/2} (f - \bar{f}) \rangle \,.
\end{equation*}
Since $\mB^*\mB$ is nonnegeative we conclude,  as in the case of reversible perturbations, that $\sigma_0(f) \le \sigma(f)$.  \qed

\subsection{Large deviations.}
Finally we turn to large deviations. Let us assume that  all the processes involved satisfy a large deviation principle
for the empirical measure
\begin{equation*}
\mu_T \,=\, \frac{1}{T} \int_{0}^{T} \delta_{X(s)} \, ds
\end{equation*}
with a rate function $I(\mu)$ which is given by Donsker-Varadhan formula
\begin{equation}\label{BV}
I ( \mu) \,=\, -\inf_{ u >0, u \in D(\mL)} \int  \frac{ \mL u}{u} \, d \mu  \,.
\end{equation}
Symbolically, we write
\begin{equation*}
\mathbb{P} \left\{ \mu_{t}  \approx \mu \right\} \asymp e^{- t I(\mu)}
\end{equation*}
where $\asymp$ denotes logarithmic equivalence and the rate function $I(\mu)$
quantifies the exponential rate at which the random measure $\mu_t$ converges to $\pi$.  Clearly, the larger $I$ is, the faster
the convergence occurs.

In particular, the formal definition is as follows. Let $E$ be a Polish space, i.e., a complete and separable metric space. Denoting by
$\mathcal{P}(E)$ the space of all probability measures on $E$,  we equip $\mathcal{P}(E)$ with the topology of
weak convergence, which makes $\mathcal{P}(E)$  metrizable and a Polish space.

\begin{definition}\label{Def:LDP}
Consider a sequence of random probability measures $\{\mu_{t}\}$. The family $\{\mu_{t}\}$ is said to satisfy a large deviations principle (LDP) with rate
 function (equivalently action functional) $I:\mathcal{P}(E)\mapsto [0,\infty]$ if the following conditions hold:
\begin{itemize}
\item{For all open sets $O\subset \mathcal{P}(E)$, we have
\[
\liminf_{t\rightarrow\infty}\frac{1}{t}\log \mathbb{P}\left\{\mu_{t}\in O\right\}\geq -\inf_{\mu\in O}I(\mu)
\] }
\item{For all closed sets $F\subset \mathcal{P}(E)$, we have
\[
\limsup_{t\rightarrow\infty}\frac{1}{t}\log \mathbb{P}\left\{\mu_{t}\in F\right\}\leq -\inf_{\mu\in F}I(\mu)
\] }
\item{The level sets $\{\mu: I(\mu)\leq M\}$ are compact in $\mathcal{P}(E)$ for all $M<\infty$.}
\end{itemize}
\end{definition}

We also assume that in the reversible case we have the following
\begin{equation}\label{DVreversible}
I_o ( \mu) \,=\,  \left\langle \left(\frac{d \mu}{d\pi}\right)^{1/2},  (- \mL_0) \left(\frac{d \mu}{d\pi}\right)^{1/2} \right\rangle
\end{equation}
This has been proved in various cases, for discrete state space Markov chains and diffusions with smooth transition probability densities in \cite{DonskerVaradhan1975}.  The case of general jump processes is only partially understood, see \cite{DupuisYufei2013} for the reversible case where a generalization of \eqref{DVreversible} is proved. Using this we obtain Lemma \ref{L:LargeDeviationsMeasures}.

\begin{lemma}{\bf [Large deviations for empirical measures].}\label{L:LargeDeviationsMeasures} Let us consider measures $\mu\in\mathcal{P}(E)$ such that $\left(d\mu/d\pi\right)^{1/2}\in D(\mathcal{L}_{0})$. Let $I(\mu)$ be the rate function associated with $\mL$ and $I_{o}(\mu)$ the rate function associated with $\mL_{0}$. We have
\begin{equation*}
I(\mu) \ge I_o(\mu)
\end{equation*}
\end{lemma}

\proof For reversible perturbations by negative definite operators $\mS$ this follows directly from the formula \eqref{DVreversible}.  For nonreversible perturbations, we can simply take $u=u_{0}$ in (\ref{BV}), where
\begin{equation*}
u_0 = \left(\frac{d \mu}{d \pi}\right)^{1/2}.
\end{equation*}

Then we have
\begin{equation*}
I(\mu) \ge  \int  \frac{ (-\mL_0- \mA) u_0}{u_0} \, d \mu  \,=\, \int  u_0  (-\mL_0- \mA) u_0 \, d \pi  \,=\, \int  u_0  (-\mL_0) u_0 \, d \pi\,=\, I_o ( \mu).
\end{equation*}
\qed

For $f \in {\mathcal C}(E)$ the contraction principle implies that the ergodic average $\frac{1}{t} \int_0^t f(X_s) ds$
satisfies a large deviation  principle with rate function
\begin{equation*}
\tilde{I}_{f}(\ell)=\inf_{\mu\in\mathcal{P}(E)}\left\{I(\mu): \left<f,\mu\right>=\ell\right\} \,.
\end{equation*}

It is a nonnegative convex function with a minimum $\tilde{I}_{f} (\bar{f})=0$ at $\ell= \bar{f}$ and it
is finite for the range of $f$, i.e. on the open interval $( \min_x f(x)\,, \max_x f(x))$.  One uses the  informal notation
$\mathbb{P} \left\{  \frac{1}{t} \int_0^t f(X_s)   \approx \ell \right\} \asymp e^{- t \tilde{I}_{f} (\ell)}$ to express that
\begin{equation*}
\lim_{\epsilon \to 0} \lim_{t \to \infty} t \log \mathbb{P} \left\{ \frac{1}{t} \int_0^t f(X_s)\, ds  \in (\ell - \epsilon, \ell + \epsilon) \right\} \,=\,
\tilde{I}_{f} (\ell)
\end{equation*}
if $\ell$ is in the range of $f$.

A Markov process whose rate function $\tilde{I}_{f}(\ell)$ is higher means that its ergodic average converges faster to its equilibrium value. In fact, we have the following lemma.

\begin{lemma}{\bf [Large deviations for observables].}\label{L:LargeDeviationsObservable}
 Consider $f \in \mathcal{C}^{(\alpha)}(E)$ and $\ell \in ( \min_x f(x), \max_x f(x))$ with $\ell \not= \int f d \pi$.
Then we have
\begin{equation*}
{\tilde I}_{f} (\ell) \ge {\tilde I}_{f,0}(\ell) \,,
\end{equation*}
where $\tilde{I}_{f,0}(\ell)=\inf_{\mu\in\mathcal{P}(E)}\left\{I_{o}(\mu): \left<f,\mu\right>=\ell\right\} \,.$
\end{lemma}

\begin{proof}
By definition $\tilde{I}_{f} (\ell)$ is the infimum of  $I(\mu)$ over all $\mu\in\mathcal{P}(E)$ such that
$\left<f,\mu\right>=\ell$. It easily follows by the affine form of the constraint $\left<f,\mu\right>=\ell$ in the definition of $\tilde{I}_{f} (\ell)$, that $\tilde{I}_{f} (\ell)$ is a convex functional. Then, convexity and Lemma \ref{L:LargeDeviationsMeasures} trivially imply the statement of the lemma.
\end{proof}

We conclude this section by mentioning that Lemma \ref{L:AsymptoticVariance} can be seen as a simple consequence of Lemma \ref{L:LargeDeviationsObservable}. Indeed, it is well known in the large deviations literature, see for example \cite{Hollander2000}, that the asymptotic variance is inversely proportional to the second derivative of the large deviations rate function evaluated at $\ell=\bar{f}$, i.e.,
\[
\sigma^{2}(f)=\frac{1}{2\tilde{I}^{''}_{f}(\bar{f})}.
\]

Then Lemma \ref{L:LargeDeviationsObservable} and convexity of the rate function immediately imply the statement of Lemma \ref{L:AsymptoticVariance}. In addition to that, as we shall see in Section \ref{S:LDPanalysisDiffusions},
a more careful analysis of the large deviations rate function reveals when there is a strict improvement in performance. It turns out that whether or not one has strict improvement in performance is related to the solution of a specific nonlinear Poisson equation.   We note here that the Poisson equation that we derive is reminiscent of Poisson equations that have appeared in the literature in the analysis of MCMC algorithms, see for example Chapter 17 of \cite{MeynTweedie}. In this paper, we see that the specific  Poisson equation that we derive, characterizes when irreversible perturbations lead to strict improvement in performance.

\section{Large deviations analysis of irreversible perturbation for Markov chains.}\label{S:LDPanalysisMarkovChain}
We consider a finite state aperiodic irreducible Markov chain with transition probability kernel $k_{0}(i,j)$ and invariant measure $\pi$. The generator of such a jump Markov process takes the form
\[
\mL_{0}g(i)=\sum_{j}\left[k_{0}(i,j)(g(j)-g(i))\right]
\]

Hence, the Donsker-Varadhan rate function takes the form
\[
I(\mu)=-\inf_{g>0}\left[\sum_{i}\frac{\mu(i)}{g(i)}\sum_{j}\left[k_{0}(i,j)(g(j)-g(i))\right]\right]
\]
and as it is proven in \cite{MaesNetocnyWynants2012} this can be simplified to
\[
I(\mu)=\sum_{i,j}\mu(i)k_{0}(i,j)\left(1-e^{\frac{V_{0}(j)-V_{0}(i)}{2}}\right)=\sum_{i,j}\mu(i)k_{0}(i,j)- \sum_{i,j}\mu(i)k_{0}(i,j)e^{\frac{V_{0}(j)-V_{0}(i)}{2}}
\]
where $V_{0}$ is the unique solution (up to a constant) of the algebraic equation
\begin{equation}
\sum_{j}\left[k_{0}(i,j)e^{\frac{V_{0}(j)-V_{0}(i)}{2}}\mu(i)-k_{0}(j,i)e^{\frac{V_{0}(i)-V_{0}(j)}{2}}\mu(j)\right]=0, \text{ for all }i\in K.\label{Eq:AlgebraicEquationForIrreversibleMC}
\end{equation}

The last relation shows that $\kappa_{V_{0}}(i,j)=k_{0}(i,j)e^{\frac{V_{0}(j)-V_{0}(i)}{2}}$ is the transition probability density function for the Markov chain with invariant measure $\mu$. Obviously if $\mu=\pi$, the  only possible solution to (\ref{Eq:AlgebraicEquationForIrreversibleMC}) is $V_{0}(i)=\text{constant}$ for every $i\in K$. As expected, this of course means that $I(\pi)=0$. Notice that under irreducibility,  we can write
\[
I(\mu)=\sum_{i,j}\mu(i)k_{0}(i,j)- \sum_{i,j}\mu(i)\kappa_{V_{0}}(i,j)
\]
which means that the rate function can be viewed as the difference between the expected escape rates $\sum_{i,j}\mu(i)k_{0}(i,j)$ and $\sum_{i,j}\mu(i)\kappa_{V_{0}}(i,j)$. The latter naturally estimates the difference in the number of transitions per unit time in the process.

As it has been observed in Example \ref{IrreversibleMC}, if we consider a matrix $\Gamma$ that is anti-symmetric and its rows sum to zero, i.e.,
\[
\Gamma=-\Gamma^{T} \quad \text{ and for every }i\in K \quad \sum_{j}\Gamma(i,j)=0
\]
then, the Markov chain with transition probability matrix $k_{\Gamma}(i,j)=k_{0}(i,j)+\frac{1}{\pi(i)}\Gamma(i,j)$ will have the same invariant distribution $\pi$. Let us denote by $V_{\Gamma}$ the solution to  (\ref{Eq:AlgebraicEquationForIrreversibleMC}) with $k_{\Gamma}$ in place of $k_{0}$.

Our goal  is to compare the rate functions of the two Markov chains, the one with transition probability function $k_{0}(i,j)$ and the one with transition probability function $k_{\Gamma}(i,j)$. Let us denote the associated large deviations rate functions by
$I_{0}(\mu)$ and $I_{\Gamma}(\mu)$ respectively. Let us define, for a given transition rate function $k(i,j)$ and a function $V$ defined on the state space of the Markov chain, the functional
\[
\mathcal{Y}_{k}(V)=\sum_{i,j}\mu(i)k(i,j)e^{\frac{V(j)-V(i)}{2}}
\]

It is easy to see that the functional $\mathcal{Y}_{k}(V)$ is non-negative and, under the irreducibility assumption, strictly convex, \cite{MaesNetocnyWynants2012}, with respect to functions $V$ defined on the state space of the Markov chain. The unique minimum for $\mathcal{Y}_{k_{0}}$ is attained at $V=V_{0}$ whereas the unique minimum for $\mathcal{Y}_{k_{\Gamma}}$ is attained at $V=V_{\Gamma}$.

Let us prove now, using Lemma \ref{L:LargeDeviationsMeasures} that $\mathcal{Y}_{k_{0}}(V_{0})\geq\mathcal{Y}_{k_{\Gamma}}(V_{\Gamma})$. In particular, this means that the minimum value of the functional $\mathcal{Y}_{k_{\Gamma}}(\cdot)$ is below the minimum value of the functional $\mathcal{Y}_{k_{0}}(\cdot)$. This means that under irreversibility, there are more transitions per unit time in the process, which then naturally leads to faster convergence to equilibrium.

\begin{proposition}\label{P:ComparisonRateFcnDiscreteCaseIrreversible}
With the notation above we have that
\begin{align*}
I_{\Gamma}(\mu)-I_{0}(\mu)
&=\mathcal{Y}_{k_{0}}(V_{0})-\mathcal{Y}_{k_{\Gamma}}(V_{\Gamma})\geq 0\nonumber
\end{align*}
\end{proposition}

\begin{proof}[Proof of Proposition \ref{P:ComparisonRateFcnDiscreteCaseIrreversible}]
We have the following computations
\begin{align}
I_{\Gamma}(\mu)-I_{0}(\mu)&=\sum_{i,j}\mu(i)k_{\Gamma}(i,j)\left(1-e^{\frac{V_{\Gamma}(j)-V_{\Gamma}(i)}{2}}\right)-\sum_{i,j}\mu(i)k_{0}(i,j)\left(1-e^{\frac{V_{0}(j)-V_{0}(i)}{2}}\right)\nonumber\\
&=\sum_{i,j}\mu(i)\left(k_{0}(i,j)+\frac{1}{\pi(i)}\Gamma(i,j)\right)\left(1-e^{\frac{V_{\Gamma}(j)-V_{\Gamma}(i)}{2}}\right)-\sum_{i,j}\mu(i)k_{0}(i,j)\left(1-e^{\frac{V_{0}(j)-V_{0}(i)}{2}}\right)\nonumber\\
&=\sum_{i,j}\mu(i)\left(k_{0}(i,j)+\frac{1}{\pi(i)}\Gamma(i,j)-k_{0}(i,j)\right) \nonumber\\ &\qquad+\left[\sum_{i,j}\mu(i)k_{0}(i,j)e^{\frac{V_{0}(j)-V_{0}(i)}{2}}-\sum_{i,j}\mu(i)\left(k_{0}(i,j)+\frac{1}{\pi(i)}\Gamma(i,j)\right)e^{\frac{V_{\Gamma}(j)-V_{\Gamma}(i)}{2}}\right]\nonumber\\
&=\sum_{i}\frac{\mu(i)}{\pi(i)}\sum_{j}\Gamma(i,j) +\nonumber\\ &\qquad+\left[\sum_{i,j}\mu(i)k_{0}(i,j)e^{\frac{V_{0}(j)-V_{0}(i)}{2}}-\sum_{i,j}\mu(i)\left(k_{0}(i,j)+\frac{1}{\pi(i)}\Gamma(i,j)\right)e^{\frac{V_{\Gamma}(j)-V_{\Gamma}(i)}{2}}\right]\nonumber\\
&=\left[\sum_{i,j}\mu(i)k_{0}(i,j)e^{\frac{V_{0}(j)-V_{0}(i)}{2}}-\sum_{i,j}\mu(i)\left(k_{0}(i,j)+\frac{1}{\pi(i)}\Gamma(i,j)\right)e^{\frac{V_{\Gamma}(j)-V_{\Gamma}(i)}{2}}\right]\nonumber\\
&=\mathcal{Y}_{k_{0}}(V_{0})-\mathcal{Y}_{k_{\Gamma}}(V_{\Gamma})
\end{align}
In the last computation we used the fact that $\sum_{j}\Gamma(i,j)=0$. Since, by Lemma \ref{L:LargeDeviationsMeasures}, we have that $I_{\Gamma}(\mu)\geq I_{0}(\mu)$, we conclude the proof of the proposition.
\end{proof}

\section{Large deviations analysis of reversible and irreversible perturbation for diffusions.}\label{S:LDPanalysisDiffusions}

It turns out that, in the case of diffusion processes,  the large deviations criterion can give more concrete information on how much improvement one gets by reversible and irreversible perturbations.  Let us consider the overdamped Langevin equation
\begin{equation}\label{Langevin1}
dX_0(t) \,=\,  -   \nabla U(X_0(t)) dt  + \sqrt{2 T} dB(t)
\end{equation}

In Examples \ref{Ex:ReversibleDiffusion} and \ref{Ex:IrreversibleDiffusion} we proposed specific reversible and irreversible perturbations of the infinitesimal
 generator of (\ref{Langevin}) that, based on Lemmas \ref{L:SpectralGap}, \ref{L:AsymptoticVariance},  \ref{L:LargeDeviationsMeasures} and \ref{L:LargeDeviationsObservable}, lead to faster convergence to equilibrium, irrespectively of which  performance criteria is being used. Our goal in this section is to characterize the improvement in sampling in more precise terms. We use the large deviations formalism for empirical measures.

As it turns out, we can write down how much the rate function increases when a reversible or an irreversible perturbation is performed. Based on the corresponding formula we can then characterize exactly when there is a strict increase in performance. The special case of irreversible perturbations of diffusions from Example \ref{Ex:IrreversibleDiffusion} has been extensively studied in \cite{HwangMaSheu2005} based on spectral gap criteria and recently on \cite{ReyBelletSpiliopoulos2014,ReyBelletSpiliopoulos2014b} based on the asymptotic variance and large deviations rate function criteria. We refer the interested reader to \cite{HwangMaSheu2005,ReyBelletSpiliopoulos2014,ReyBelletSpiliopoulos2014b} for further details and for numerical results. In this section we compare how reversible and irreversible perturbations for general Markov processes compare via the lens of large deviations theory. The results of \cite{HwangMaSheu2005,ReyBelletSpiliopoulos2014,ReyBelletSpiliopoulos2014b} are then essentially recovered as a special case of the general theory of this paper.

Let us start our analysis with a very general result on the large deviations principle for the invariant measure of diffusion processes. In order to avoid technical issues we shall restrict our discussion to diffusion taking values on a $d-$dimensional compact Riemannian manifold $E$ of class $C^{3}$ without boundary.  In particular, we have the following general theorem.

\begin{theorem}\label{Th:Gartner} Consider the SDE on $E$ with infinitesimal generator
\[
\mathcal{L} = \frac{1}{2}\nabla\cdot a(x)\nabla + b(x) \nabla
\]
with $b_{i},a_{i,j} \in {\mathcal C}^1(E)$, $a(x)$ being strictly positive. Let $\mu \in \mathcal{P}(E)$, where  $\mu(dx) =p(x)dx$ is a measure with positive density $p \in \mathcal{C}^{(2 + \alpha)}(E)$ for some $\alpha >0$. The Donsker-Vardhan rate function $I(\mu)$ takes the form
\begin{equation}\label{Eq:GartnerFormula1}
I(\mu)=\frac{1}{8}\int_{E}\frac{\nabla p(x)a(x)\nabla p(x)}{p^{2}(x)}d\mu(x)-\frac{1}{2}\int_{E}\frac{b(x)\nabla p(x)}{p(x)}d\mu(x)+\frac{1}{2}\int_{E}\nabla\phi(x)a(x)\nabla\phi(x) d\mu(x)
\end{equation}
where $\phi$ is the unique (up to constant) solution of the equation
\begin{equation}\label{Eq:GartnerFormula1Constraint}
\text{div}\left[p(x)\left(b(x)+a(x)\nabla\phi(x)\right)\right]=0.
\end{equation}
\end{theorem}
\begin{proof}
The proof of this theorem follows the same steps as that of Lemma 3.2 in \cite{ReyBelletSpiliopoulos2014} using the general results of  G\"{a}rtner in \cite{Gartner1977}. Thus, the details are omitted.
\end{proof}

In the case of equation (\ref{Langevin}), i.e., when $b(x)=-\nabla U(x)$ is a gradient and $a(x)=2T I$, then $\phi(x)=\frac{1}{2T}U(x)+\textrm{constant}$
and we get
\begin{equation}\label{explicitreversible}
I_{o}(\mu)=\frac{T}{4}\int_{E}\left|\frac{\nabla p(x)}{p(x)}+\frac{1}{T}\nabla U(x)\right|^{2}d\mu(x)
\end{equation}
which is the usual explicit formula for the rate function in the reversible case.

In this section we want to compare the rate function for the baseline case (\ref{Langevin}) with that of the reversible perturbation of Example  \ref{Ex:ReversibleDiffusion} and that of the irreversible perturbation of Example \ref{Ex:IrreversibleDiffusion}.

For notational convenience, let us denote by
\begin{enumerate}
\item{$I_{\Sigma}(\mu)$ the rate function for the diffusion of Example \ref{Ex:ReversibleDiffusion}, i.e., when $a(x)=2T\Sigma(x)$ and $b(x)=-\Sigma(x)\nabla U(x)$,}
 \item{$I_{C}(\mu)$ the rate function for the diffusion of Example \ref{Ex:IrreversibleDiffusion}, i.e., when $a(x)=2T I$ and $b(x)=-\nabla U(x)+C(x)$, and}
  \item{ $I_{\Sigma,C}(\mu)$ the rate function for the diffusion  when $a(x)=2T\Sigma(x)$ and $b(x)=-\Sigma(x)\nabla U(x)+C(x)$.}
\end{enumerate}
Clearly, using this notation, the rate function for the reference case, (\ref{explicitreversible}), is $I_{o}(\cdot)=I_{Id,O}(\cdot)$.

Propositions \ref{T:measure1}, \ref{T:measure2} and \ref{T:measure3} summarize  the increase of the Donsker-Varadhan rate functions for empirical measures based on reversible and
irreversible perturbations.  For presentation purposes, the proofs of these results is given at the end of the section. Moreover, based on these results we can then prove that the rate function for the empirical  average of a given observable also increases under the suggested reversible and irreversible perturbations. This is Theorem \ref{T:RateFcnObservables}. Conditions, guaranteeing strict improvement in performance are also provided. It turns out that whether or not one has strict improvement in performance is related to the solution of a specific nonlinear Poisson equation.

\begin{proposition}
\label{T:measure1} Assume that the matrix $\Sigma(x) \not= I$ is such that $\Sigma(x)-I$ is nonnegative definite. For any $\mu \in \mathcal{P}(E)$ we have $I_{\Sigma}(\mu) \ge
I_0(\mu)$. If  $\mu(dx) =p(x)dx$ is a measure with positive density $p \in \mathcal{C}^{(2 + \alpha)}(E)$ for some $\alpha >0$ and $\mu \not = \pi$
then we have
\begin{equation*}
I_{\Sigma}(\mu) - I_{o}(\mu) = \frac{T}{4}\int_{E}\left(\frac{\nabla p(x)}{p(x)}+\frac{1}{T}\nabla U(x)\right)^{T}(\Sigma(x)-I)\left(\frac{\nabla p(x)}{p(x)}+\frac{1}{T}\nabla U(x)\right)d\mu(x)\geq 0
\end{equation*}
Moreover we have that if $p(x)>0$ everywhere and $\Sigma(x)-I$ is strictly positive everywhere, then $I_{\Sigma}(\mu) > I_o(\mu)$.
\end{proposition}

\begin{proposition}\label{T:measure2} Assume that the vector field $C(x) \not=0$ is such that ${\rm div}(C(x) e^{-U(x)/T}) =0$ and the matrix $\Sigma(x)$ is strictly positive definite. For any $\mu \in \mathcal{P}(E)$ we have $I_{\Sigma, C}(\mu) \ge
I_{\Sigma}(\mu)$. If  $\mu(dx) =p(x)dx$ is a measure with positive density $p \in \mathcal{C}^{(2 + \alpha)}(E)$ for some $\alpha >0$ and $\mu \not = \pi$
then we have
\begin{equation*}
I_{\Sigma,C}(\mu) - I_{\Sigma}(\mu) = 4T\int_{E}\left(\frac{1}{2}\nabla \phi(x)-\frac{1}{4T}\nabla U(x)\right)^{T}\Sigma(x)\left(\frac{1}{2}\nabla \phi(x)-\frac{1}{4T}\nabla U(x)\right)d\mu(x)\geq 0
\end{equation*}
where $\phi$ is the unique solution (up to a constant) of the equation
\begin{equation*}
{\textrm div}\left[p(x)\left(-\Sigma(x)\nabla U(x)+C(x)+2T\Sigma(x)\nabla \phi(x)\right)\right]=0.
\end{equation*}

Moreover, if the positive density  $p(x)$ satisfies $\text{div}\left(p(x)C(x)\right)\neq 0$, then we have $I_{\Sigma,C}(\mu) > I_{\Sigma}(\mu)$.
If $p(x)$ is such that $\text{div}\left(p(x)C(x)\right)=0$, then it has the form $p(x) = e^{2 G(x)}$ where $G$ is such that $G+U$ is an invariant quantity for the
vector field $C$ (i.e., $C \nabla (G+U) =0$).
\end{proposition}

Clearly, if we set $\Sigma(x)=I$, then Proposition \ref{T:measure2} shows that for the irreversible perturbation of Example \ref{Ex:IrreversibleDiffusion} one has
\begin{equation*}
I_{C}(\mu) - I_{o}(\mu) = 4T\int_{E}\left(\frac{1}{2}\nabla \phi(x)-\frac{1}{4T}\nabla U(x)\right)^{T}\left(\frac{1}{2}\nabla \phi(x)-\frac{1}{4T}\nabla U(x)\right)d\mu(x)\geq 0.
\end{equation*}

This is nothing else but Theorem 2.2 in \cite{ReyBelletSpiliopoulos2014}. As a matter of fact \cite{ReyBelletSpiliopoulos2014,ReyBelletSpiliopoulos2014b} study in detail this special case via the lens of large deviations theory. We refer the interested reader to these articles for further details on this special case and related numerical simulation results. Next, in Proposition \ref{T:measure3} we investigate the situation where one performs both reversible and irreversible perturbations.

\begin{proposition}\label{T:measure3} Assume that the vector field $C(x) \not=0$ is such that ${\rm div}(C(x) e^{-U(x)/T}) =0$ and the matrix $\Sigma(x)-I$ is nonnegative definite. For any $\mu \in \mathcal{P}(E)$ we have $I_{\Sigma, C}(\mu) \ge
I_{o}(\mu)$. If  $\mu(dx) =p(x)dx$ is a measure with positive density $p \in \mathcal{C}^{(2 + \alpha)}(E)$ for some $\alpha >0$ and $\mu \not = \pi$
then we have
\begin{align*}
I_{\Sigma,C}(\mu) - I_{o}(\mu) &= \frac{T}{4}\int_{E}\left(\frac{\nabla p(x)}{p(x)}+\frac{1}{T}\nabla U(x)\right)^{T}(\Sigma(x)-I)\left(\frac{\nabla p(x)}{p(x)}+\frac{1}{T}\nabla U(x)\right)d\mu(x) \nonumber\\
&+4T\int_{E}\left(\frac{1}{2}\nabla \phi(x)-\frac{1}{4T}\nabla U(x)\right)^{T}\Sigma(x)\left(\frac{1}{2}\nabla \phi(x)-\frac{1}{4T}\nabla U(x)\right)d\mu(x)  \nonumber\\
&\geq 0\,.
\end{align*}
where $\phi$ is the unique solution (up to a constant) of the equation
\begin{equation*}
{\textrm div}\left[p(x)\left(-\Sigma(x)\nabla U(x)+C(x)+2T\Sigma(x)\nabla \phi(x)\right)\right]=0.
\end{equation*}

Moreover, if the positive density  $p(x)$ satisfies $\text{div}\left(p(x)C(x)\right)\neq 0$ and  $,\Sigma(x), \Sigma(x)-I$ are strictly positive definite, then we have $I_{\Sigma,C}(\mu) > I_{o}(\mu)$.
If $p(x)$ is such that $\text{div}\left(p(x)C(x)\right)=0$, then it has the form $p(x) = e^{2 G(x)}$ where $G$ is such that $G+U$ is an invariant quantity for the
vector field $C$ (i.e., $C \nabla (G+U) =0$).
\end{proposition}

Notice that the correction term in Proposition \ref{T:measure3} is the sum of the correction terms from Propositions \ref{T:measure1} and \ref{T:measure2}. This comes to no surprise, as the set-up of Proposition \ref{T:measure3} is that of both reversible and irreversible perturbation. 

Based on these results we then study the impact of these perturbations on the large deviations for the estimator $f_{t}=\frac{1}{t} \int_0^t f(X_s) ds$ itself. For $f \in {\mathcal C}(E)$ contraction principle implies that the ergodic average $\frac{1}{t} \int_0^t f(X_s) ds$
satisfies a large deviation  principle with rate function
\begin{equation*}
\tilde{I}_{f}(\ell)=\inf_{\mu\in\mathcal{P}(E)}\left\{I(\mu): \left<f,\mu\right>=\ell\right\}.
\end{equation*}

As we remarked in Section \ref{S:MainLemmas}, a Markov process whose rate function $\tilde{I}_{f}(\ell)$ is higher means that its ergodic average converges faster to its equilibrium value, in the sense that the rate of the exponential convergence is faster.

By general principles, see for example \cite{Gartner1977}, the rate function $\tilde{I}_{f} (\ell)$ is given by the Legendre transform
$\tilde{I}_{f} (\ell)= \sup_{\beta \in \mathbb{R}} (\ell \beta - \lambda(\beta f))$
 where
\begin{equation}\label{momentgeneratingfunction}
\lambda(\beta f) = \lim_{t \to \infty} \frac{1}{t} \log \mathbb{E}_x \left[e^{ \beta \int_0^t f(X_s) ds} \right]\,
\end{equation}

Using a Perron-Frobenius argument one can show that $\lambda(\beta f)$ is maximal eigenvalue of the operator
${\mL} + \beta f$ and that  $\lambda(\beta)$ is a smooth (real-analytic) function of $\beta$ and hence
\begin{equation*}
\tilde{I}_{f} (\ell)= \ell \widehat{\beta} - \lambda(\widehat{\beta} f )
\end{equation*}
where $\widehat{\beta}=\widehat{\beta}(\ell)$ is the unique  solution of  $\frac{d}{d\beta} \lambda(\beta f) = \ell$.

We denote by $\tilde{I}_{f,\Sigma,C}(\ell)$, $\tilde{I}_{f,\Sigma}(\ell)$, $\tilde{I}_{f,C}(\ell)$ and $\tilde{I}_{f,o}(\ell)$ the rate functions corresponding to $I_{\Sigma,C}(\mu)$, $I_{\Sigma}(\mu)$, $I_{C}(\mu)$ and $I_{o}(\mu)$ respectively.

\begin{theorem}\label{T:RateFcnObservables}
Consider $f \in \mathcal{C}^{(\alpha)}(E)$ and $\ell \in ( \min_x f(x), \max_x f(x))$ with $\ell \not= \int f d \pi$. Fix a vector field $C$ such that ${\rm div}(C(x) e^{-U(x)/T}) =0$ and let $\Sigma(x)$ be such that $\Sigma(x)-I$ is nonnegative definite.
Then we have
\begin{equation*}
{\tilde I}_{f,\Sigma,C} (\ell) \ge {\tilde I}_{f,\Sigma} (\ell) \ge {\tilde I}_{f,o}(\ell) \,,
\end{equation*}

If $\Sigma(x)-I$ is strictly positive definite, and  if there exists $\ell_0$ such that for this particular field $C$, ${\tilde I}_{f,\Sigma,C} (\ell_0) = {\tilde I}_{f,\Sigma}(\ell_0)$ or ${\tilde I}_{f,\Sigma,C} (\ell_0) = {\tilde I}_{f,o}(\ell_0)$ then we must have
\begin{equation}
\widehat{\beta}(\ell_0) f \,=\,  e^{-(G+U)}\left(\mathcal{L}_{0}+\mS\right)e^{G+U} \,,\label{Eq:ConditionObservable}
\end{equation}
where $G$ is such that $G+U$ is invariant under the particular vector field $C$ and $\mathcal{L}_{0}+\mS$ is the infinitesimal generator of the process given in Example \ref{Ex:ReversibleDiffusion}.
\end{theorem}

We conclude this section with the proofs of Propositions \ref{T:measure1}, \ref{T:measure2}, \ref{T:measure3} and Theorem \ref{T:RateFcnObservables}.

\begin{proof}[Proof of Proposition \ref{T:measure1}.]
Consider the general situation of Theorem \ref{Th:Gartner} with $a(x)=2T\Sigma(x)$ and $b(x)=-\Sigma(x)\nabla U(x)$. We then have
\begin{align*}
I_{\Sigma}(\mu)&=\frac{2T}{8}\int_{E}\frac{\nabla p(x)\Sigma(x)\nabla p(x)}{p^{2}(x)}d\mu(x)+\frac{1}{2}\int_{E}\frac{\nabla p(x)\Sigma(x)\nabla U(x)}{p(x)}d\mu(x)+T\int_{E}\nabla\phi(x)\Sigma(x)\nabla\phi(x) d\mu(x),
\end{align*}
where due to reversibility we have $\phi(x)=\frac{1}{2T}U(x)+\text{constant}$. Notice that $I_{o}(\mu)$ is nothing else but $I_{\Sigma}(\mu)$ with $\Sigma(x)=I$. Taking then, the difference $I_{\Sigma}(\mu)-I_{o}(\mu)$ and doing some straightforward algebra, we obtain the statement of the proposition. Clearly,  $I_{\Sigma}(\mu)-I_{o}(\mu)>0$ if $\Sigma(x)-I$ is strictly positive definite.
\end{proof}

\begin{proof}[Proof of Proposition \ref{T:measure2}.]
Considering the general situation of Theorem \ref{Th:Gartner}, we obtain for the difference
\begin{align*}
I_{\Sigma,C}(\mu) - I_{\Sigma}(\mu)&=\int_{E}\left[T\left(\nabla\phi(x)\Sigma(x)\nabla\phi(x)\right)-\frac{1}{4T}\nabla U(x)\Sigma(x)\nabla U(x)-\frac{1}{2} \frac{C(x)\nabla p(x)}{p(x)}\right]d\mu(x)
\end{align*}

Using the condition
$\textrm{div}\left(C(x)e^{-U(x)/T}\right)=0$, which can be rewritten as $\textrm{div}C(x)=T^{-1} C(x)\nabla U(x)$, and integrating by parts we get for the last term of the last display
\begin{align*}
\int_{E}\frac{C(x)\nabla p(x)}{p(x)}d\mu(x) &=\int_{E} C(x)\nabla p(x)dx=-\int_{E} \textrm{div}C(x) p(x)dx=-\int_{E} \textrm{div}C(x) d\mu(x)\nonumber\\
&=-\int_{E} \frac{1}{T} C(x)\nabla U(x) d\mu(x).
\end{align*}

Thus we have obtained
\begin{align*}
I_{\Sigma,C}(\mu) - I_{\Sigma}(\mu)&=\int_{E}\left[T\left(\nabla\phi(x)\Sigma(x)\nabla\phi(x)\right)-\frac{1}{4T}\nabla U(x)\Sigma(x)\nabla U(x)+\frac{1}{2T} C(x)\nabla U(x)\right]d\mu(x)
\end{align*}

Recall now that $\phi(x)$ is the unique solution, up to constants, of the equation
\begin{equation*}
{\textrm div}\left[p(x)\left(-\Sigma(x)\nabla U(x)+C(x)+2T\Sigma(x)\nabla \phi(x)\right)\right]=0.
\end{equation*}

Its weak form reads as follows
\begin{equation}
 \int_{E}\nabla g(x)\left[2T\Sigma(x)\nabla \phi(x)-\Sigma(x)\nabla U(x)+C(x)\right]d\mu(x)=0, \quad \textrm{for every }g\in\mathcal{C}^{1}(E)\label{Eq:WeakForm1}
\end{equation}
and we can pick freely $g\in\mathcal{C}^{1}(E)$. Let us  first choose $g(x)=\frac{1}{2}\phi(x)+\frac{1}{4T}U(x)$. Then, (\ref{Eq:WeakForm1}) gives
\begin{equation*}
\int_{E}\left[T\left(\nabla\phi(x)\Sigma(x)\nabla\phi(x)\right)-\frac{1}{4T}\nabla U(x)\Sigma(x)\nabla U(x)\right]d\mu(x)=-\int_{E}C(x)\left(\frac{1}{2}\nabla\phi(x)+\frac{1}{4T}\nabla U(x)\right)d\mu(x)
\end{equation*}
and thus, we obtain
\begin{align}
I_{\Sigma,C}(\mu) - I_{\Sigma}(\mu)&=\int_{E}C(x)\left(-\frac{1}{2}\nabla\phi(x)+\frac{1}{4T}\nabla U(x)\right)d\mu(x)\label{Eq:RepDifference2}
\end{align}

Choosing then $g(x)=\frac{1}{2}\phi(x)-\frac{1}{4T}U(x)$ and we get from (\ref{Eq:WeakForm1}) and the latter display
\begin{equation*}
I_{\Sigma,C}(\mu) - I_{\Sigma}(\mu)=4T\int_{E}\left(\frac{1}{2}\nabla \phi(x)-\frac{1}{4T}\nabla U(x)\right)^{T}\Sigma(x)\left(\frac{1}{2}\nabla \phi(x)-\frac{1}{4T}\nabla U(x)\right)d\mu(x)
\end{equation*}
which is the statement of the proposition. It is clear that $I_{\Sigma,C}(\mu) - I_{\Sigma}(\mu)\geq 0$. If $\Sigma(x)$ is strictly positive definite and $\mu$ possesses a strictly positive density, it is clear that
$I_{\Sigma,C}(\mu) - I_{\Sigma}(\mu)= 0$ if and only if $\text{div}\left(pC\right)= 0$. In other words, $I_{\Sigma,C}(\mu) - I_{\Sigma}(\mu)> 0$ if and only if $\text{div}\left(pC\right)\neq 0$ and $\Sigma(x)$ is strictly positive definite. It is clear that if $\text{div}\left(p(x)C(x)\right)= 0$, then the requirement $\text{div}\left(C(x)e^{-U(x)/T}\right)= 0$ implies that $p$ can be written as $p(x)=e^{G(x)}$ with $C(x)\nabla(G(x)+U(X))=0$.
\end{proof}

\begin{proof}[Proof of Proposition \ref{T:measure3}.]
We write
\begin{align*}
I_{\Sigma,C}(\mu) - I_{o}(\mu)&=\left[I_{\Sigma}(\mu) - I_{o}(\mu)\right]+\left[I_{\Sigma,C}(\mu) - I_{\Sigma}(\mu)\right]
\end{align*}

Notice that the first term on the right hand side of the last display is the difference $I_{\Sigma}(\mu) - I_{o}(\mu)$ from Proposition \ref{T:measure1}, whereas the second term is the difference $I_{\Sigma,C}(\mu) - I_{\Sigma}(\mu)$ from Proposition \ref{T:measure2}. This concludes the proof of the proposition.
\end{proof}

\begin{proof}[Proof of Theorem \ref{T:RateFcnObservables}.]

 Analogously to Proposition 4.1 of \cite{ReyBelletSpiliopoulos2014} we have that for each one of the infimization problems defining ${\tilde I}_{f,\Sigma} (\ell)$, ${\tilde I}_{f,\Sigma,C} (\ell)$ and ${\tilde I}_{f,o} (\ell)$ there is a corresponding infimizing measure (different for each case)
 $\mu^*(dx) = p^{*}(x)dx$  with $p^{*}(x) > 0$ and $p^{*}(x)  \in \mathcal{C}^{(2+\alpha)}(E)$ that attains the infimum. For example, in the case of only a reversible perturbation we have that
 \begin{equation*}
 \tilde{I}_{f,\Sigma}(\ell) = I_{\Sigma}(\mu^*) \,.
 \end{equation*}

 Then, a straightforward contradiction argument that is based on Propositions \ref{T:measure1}, \ref{T:measure2} and \ref{T:measure3} leads to the proof of the statement ${\tilde I}_{f,\Sigma,C} (\ell) \ge {\tilde I}_{f,\Sigma} (\ell) \ge {\tilde I}_{f,o}(\ell)$.

The derivation of the PDE (\ref{Eq:ConditionObservable}) that characterizes the situation where the rate function does not increase goes as follows. It can be seen that $p^{*}(x)$ is the invariant density corresponding to the infinitesimal generator ${\mathcal L} + \nabla\phi_{ \widehat{\beta}} \cdot \nabla$ where $e^{\phi_{ \widehat{\beta}}}$ is the eigenfunction associated to the eigenvalue $\lambda(\widehat{\beta} f)$ defined in (\ref{momentgeneratingfunction}) for the operator $\mL+\beta f$. We recall that $\widehat{\beta}=\widehat{\beta}(\ell)$ is the unique  solution of  $\frac{d}{d\beta} \lambda(\beta f) = \ell$. Due to the dependence of $p^{*}(x)$ on  $\widehat{\beta}$, let us write $p_{\widehat{\beta}}(x)$ for $p^{*}(x)$. In other words, $p_{\widehat{\beta}}(x)$  should satisfy
 \begin{equation}\label{l1}
({\mathcal L} + \nabla\phi_{ \widehat{\beta}} \cdot \nabla)^* p_{ \widehat{\beta}} =0 \,,
\end{equation}
where $e^{\phi_{\widehat{\beta}}(x)}$ is the eigenfunction associated to the eigenvalue $\lambda(\widehat{\beta} f)$, i.e.,
\begin{equation}\label{l2}
({\mathcal L} + \widehat{\beta} f) e^{\phi_{ \widehat{\beta}}} = \lambda(\widehat{\beta} f) e^{\phi_{ \widehat{\beta}}} \,
\end{equation}
 If the rate function does not increase, then one should have that ${\rm div}( C p_{\widehat{\beta}}) =0$. Since ${\rm div}( C p_{\widehat{\beta}}) =0$ we have that in fact
\[
(\mL_{0}+\mS+ \nabla\phi_{ \widehat{\beta}} \cdot \nabla)^{*}p_{ \widehat{\beta}}=0.
\]
Since  $\mL_{0}+\mS+ \nabla\phi_{ \widehat{\beta}} \cdot \nabla$ is the generator of a
reversible ergodic Markov process, we then obtain that
$p_{\widehat{\beta}} = e^{ (\phi -U) + const}$. Thus, $\phi = G+U$ and
$C \cdot \nabla \phi=0$ and (\ref{l2}) reduces to
\begin{equation}\label{l3}
({\mathcal L}_{0}+\mS + \widehat{\beta} f) e^{\phi_{ \widehat{\beta}}} = \lambda(\widehat{\beta} f) e^{\phi_{ \widehat{\beta}}} \,
\end{equation}
Since changing $f$ into $f+c$ leaves $\phi$ unchanged, but changes $\lambda(\widehat{\beta} f)$ to $\lambda(\widehat{\beta} f)+\widehat{\beta}c $, we get that the equation $({\mathcal L}_{0}+\mS + \widehat{\beta} f) e^{\phi_{ \widehat{\beta}}} = \lambda(\widehat{\beta} f) e^{\phi_{ \widehat{\beta}}}$ implies (\ref{Eq:ConditionObservable}).
 \end{proof}

\section{Conclusions}\label{S:Conclusions}
In this paper we have demonstrated in a very general setting that perturbations of the generator of reversible Markov processes by reversible negative definite  operators or by irreversible (anti-selfadjoint) operators that maintain the invariant measure lead to improvement in sampling. In particular, we have shown that spectral gap decreases, the asymptotic variance of trajectory time averages decreases and the large deviations rate function that controls the decay rate of the tail distribution of the estimator increases. In all these three cases, we have worked with the generator of the given Markov process. Moreover, we have provided specific reversible and irreversible perturbations for cases of interest such as continuous time Markov chains, Markov jump processes as well as diffusion processes.

Clearly, there are many open questions to address here, perhaps the most important ones being optimal perturbation in different concrete cases of interest, as well as the involved numerical challenges, see \cite{ReyBelletSpiliopoulos2014, DuncanLelievrePavliotis2015}. Some preliminary results on optimal perturbations for the case of quadratic $U(x)$ (the Gaussian case) can be found in \cite{DuncanLelievrePavliotis2015, HwangMaSheu1993, LelievreNierPavliotis2012}. In addition, in most of the cases, what is guaranteed is that ergodic behavior does not become worse. It is interesting to provide concrete conditions for strict improvement, such as the ones for the diffusion case presented in \cite{ReyBelletSpiliopoulos2014} and in Section \ref{S:LDPanalysisDiffusions} of the present paper.

\end{document}